\documentclass[a4paper,11pt]{amsart}

\usepackage{comment}

\usepackage[hang,small,bf]{caption}
\usepackage[subrefformat=parens]{subcaption}
\usepackage{etoolbox}
\ifdef{\crop}{%
\usepackage[includeheadfoot,twoside=False,paperwidth=448pt,paperheight=587pt,rmargin=15pt,lmargin=15pt,tmargin=15pt,bmargin=15pt]{geometry}%
}{%
\setlength{\topmargin}{22mm}
\addtolength{\topmargin}{-1in}
\setlength{\oddsidemargin}{27mm}
\addtolength{\oddsidemargin}{-1in}
\setlength{\evensidemargin}{27mm}
\addtolength{\evensidemargin}{-1in}
\setlength{\textwidth}{156mm}
\setlength{\textheight}{240mm}
}%

\usepackage{mathtools}
\usepackage{color}
\usepackage[active]{srcltx}
\usepackage{amsmath,amsthm,amsxtra}
\usepackage{amssymb}
\usepackage{enumitem}

\usepackage[mathscr]{eucal}

\usepackage{bm}
\usepackage[all]{xy}

\usepackage{aliascnt}

\usepackage{tabularx}
\newcolumntype{C}{>{\centering\arraybackslash}X} 

\theoremstyle{plain}
\newtheorem{thm}{Theorem}[section]
\newtheorem*{thm*}{Theorem}

\newaliascnt{prop}{thm}
\newaliascnt{cor}{thm}
\newaliascnt{lem}{thm}
\newaliascnt{claim}{thm}
\newaliascnt{defn}{thm}
\newaliascnt{ques}{thm}
\newaliascnt{conj}{thm}
\newaliascnt{fact}{thm}
\newaliascnt{rem}{thm}
\newaliascnt{ex}{thm}
\newaliascnt{sett}{thm}

\newtheorem{cor}[cor]{Corollary}
\newtheorem{lem}[lem]{Lemma}
\newtheorem{claim}[claim]{Claim}
\newtheorem*{prop*}{Proposition}
\newtheorem*{cor*}{Corollary}
\newtheorem*{lem*}{Lemma}
\newtheorem*{claim*}{Claim}
\theoremstyle{definition}

\newtheorem*{defn*}{Definition}
\newtheorem*{ques*}{Question}
\newtheorem*{conj*}{Conjecture}

\newtheorem*{prob*}{Problem}

\newtheorem{rem}[rem]{Remark}
\newtheorem{ex}[ex]{Example}
\newtheorem*{fact*}{Fact}
\newtheorem*{rem*}{Remark}
\newtheorem*{ex*}{Example}
\aliascntresetthe{prop}
\aliascntresetthe{cor}
\aliascntresetthe{lem}
\aliascntresetthe{claim}
\aliascntresetthe{defn}
\aliascntresetthe{ques}
\aliascntresetthe{conj}
\aliascntresetthe{fact}
\aliascntresetthe{rem}
\aliascntresetthe{ex}
\aliascntresetthe{sett}

\usepackage{varioref}
\usepackage[a4paper]{hyperref}

\labelformat{equation}{\textnormal{(#1)}}
\labelformat{enumi}{\textnormal{(#1)}}

\def\textsectionN~{\textsection{}}

\renewcommand\phi{\varphi}
\renewcommand\epsilon{\varepsilon}
\renewcommand\leq{\leqslant}
\renewcommand\geq{\geqslant}

\newcommand{\abs}[1]{\lvert #1 \rvert}

\DeclareMathOperator{\Hom}{Hom}

\newcommand{\wid}{\mathrm{width}}

\DeclareMathOperator{\mult}{mult}

\DeclareMathOperator{\Conv}{Conv}
\DeclareMathOperator{\rank}{rk}

\DeclareMathOperator{\GL}{GL}

\DeclareMathOperator{\pr}{pr}

\def\Z{\mathbb{Z}}
\def\Q{\mathbb{Q}}
\def\R{\mathbb{R}}
\def\C{\mathbb{C}}

\def\r+{\mathbb{R}_{\geq 0}}

\def\ep{\varepsilon}

\def\r+{{\R}_{\geq 0}}
\def\q+{{\Q}_{\geq 0}}
\def\P{\mathbb{P}}

\def\*c{\C^{\times}}

\def\<{\langle}
\def\>{\rangle}

\def\lra{\Leftrightarrow}

\def\C{\mathbb {C}}

\def\Q{\mathbb {Q}}
\def\R{\mathbb {R}}

\def\Z{\mathbb {Z}}

\newcommand{\calb}{\mathcal {B}}

\newcommand{\calo}{\mathcal {O}}

\makeatletter
  
  \@addtoreset{equation}{section}
\makeatother

\DeclarePairedDelimiterX{\set}[1]{\lbrace}{\rbrace}{%
  \def\given{\;\Big\vert\;}%
  \mathopen{}#1\mathclose{}%
}

\title
[Seshadri constants, Gromov widths and generalized permutohedra]
{Seshadri constants and Gromov widths of toric varieties associated with generalized permutohedra}

\author[A.~Ito]{Atsushi~Ito}
\address{Department of Mathematics, Institute of Pure and Applied Sciences, University of Tsukuba, Tsukuba, Ibaraki 305-8571, Japan}
\email{ito-atsushi@math.tsukuba.ac.jp}

\subjclass[2020]{14C20, 52B20, 53D05}
\keywords{Seshadri constant, Gromov width, Lattice width, Generalized permutohedron}

\begin{document}

\begin{abstract}
In this paper, we study the Seshadri constants and Gromov widths of polarized toric varieties whose moment polytopes are generalized permutohedra.
We show that both invariants coincide with the lattice width of the moment polytope,
and provide an explicit formula for them in terms of the defining submodular function.
\end{abstract}

\maketitle

\section{Introduction}

The \emph{Seshadri constant} of a polarized variety $(X, L)$ at a point $p \in X $ is defined as
\begin{align}\label{eq_def_SC}
 \ep(X,L;p)=\inf_C \left\{\dfrac{(C \cdot L)}{\mult_p(C)}\right \},
\end{align}
where the infimum is taken over all reduced and irreducible curves $C$ on $X$ passing through $p$, $(C \cdot L) $ is the intersection number, and $\mult_p(C)$ is the multiplicity of $C$ at $p$.
This invariant measures the local positivity of $L$ at $p$.
We note that we can define $ \ep(X,L;p)$ in the same way
when $(X,L)$ is an \emph{$\R$-polarized variety},
that is, when $L$ is an ample $\R$-Cartier $\R$-divisor on $X$.
We refer the reader to \cite[Section 5]{MR2095471} for details.

In symplectic geometry, there exists an invariant that is closely related to Seshadri constants:
The \emph{Gromov width} $w_G(X, \omega)$ of a symplectic manifold $(X,\omega)$ of dimension $2n$ is the supremum of $\pi r^2$ 
such that the open ball $B(r) =\{ x \in \R^{2n} \mid  \|x\| < r\} $ equipped with the standard symplectic form $\omega_{\mathrm{std}} = \sum_{i=1}^n dx_{2i-1} \wedge dx_{2i}$ can be  symplectically embedded into $X$.
For a smooth $\R$-polarized variety $(X,L)$, let $\omega_L$ be a K\"ahler form on $X$ representing the Chern class $c_1(L) \in H^2(X,\R)$.
Then the inequality $\ep(X,L;p) \leq w_G(X,\omega_L)$ holds for any $p \in X$ by \cite[Corollary 2.1.D]{MR1262938}  (see also \cite[Theorem 5.1.22, Proposition 5.3.17]{MR2095471}).

\vspace{2mm}
Let $M \simeq \Z^n$ be a free abelian group of rank $n$ and $N\coloneqq\Hom(M, \Z)$ be its dual.
Set $ M_{\R} \coloneqq M \otimes_{\Z} \R $ and $N_{\R} \coloneqq N \otimes_{\Z} \R $.
For a lattice polytope $P \subset M_{\R}$ of dimension $n$,
we can construct a polarized toric variety $(X_P,L_P)$.
More generally, 
we can construct an $\R$-polarized toric variety $(X_P,L_P)$
for a \emph{quasi-rational} polytope $P \subset M_{\R}$ of dimension $n$.
Here a polytope $P \subset M_{\R}$ is called \emph{quasi-rational} if it is written as
\begin{align*}
P=\{ u \in M_{\R} \mid   \langle u,v_i \rangle \geq -a_i \text{ for } 1 \leq i \leq l\}
\end{align*}
for some $v_1, \dots, v_l \in N $ and $a_1, \dots, a_l \in \R$.

Both invariants $\ep(X,L;p), w_G(X,\omega)$ are in general difficult to compute explicitly.
No general formula is known even for toric varieties, though various results have been established (e.g., \cite{MR1696762, MR3276156} for Seshadri constants,
and  \cite{MR2210713, MR3734610, MR4251298,MR4565222, MR4717262} for Gromov widths).
In this paper,
we compute these invariants for the polarized toric variety $(X_P,L_P)$ in the case where $P$ is a \emph{generalized permutohedron}, which is defined as follows:
For a finite set $E$,  the set of all the subsets of $E$ is denoted by $2^E$.
A function $f : 2^E \to \R$ is called \emph{submodular} if
\[
f(I) + f(J) \geq f(I \cup J) + f(I \cap J)
\]
holds for any $I,J \in 2^E$.
For a submodular function $f$ with $f(\emptyset) =0$, we define the \emph{generalized permutohedron} $P_f$ as
\begin{align}\label{eq_gen_permutohedra}
P_f\coloneqq \set*{ x =(x_i)_{i \in E} \in \R^{E} \given \sum_{i \in E} x_i=f(E),\  \sum_{i \in I } x_i  \leq f(I) \text{ for any } I \in 2^{E} },
\end{align}
where $\sum_{i \in \emptyset } x_i \coloneqq 0$ by convention.
Since $f(E) - f( E \setminus \{i\}) \leq x_i \leq f(\{i\})$ for any $x \in P_f$ and $i \in E$,
the generalized permutohedron $P_f$ is bounded and hence is a quasi-rational polytope.
We note that although the original definition by A.~Postnikov \cite{MR2487491} differs from the one described above,
several equivalent definitions are now known (see \cite[Theorem 3.1.3]{MR4651496}, for instance).
We adopt this definition since we give a formula for Seshadri constants and Gromov widths using the function $f$.

We also consider  \emph{nestohedra}, 
which form a subclass of generalized permutohedra. 
Nestohedra, which were introduced independently by E.M.~Feichtner and B.~Sturmfels \cite{MR2191630} and A.~Postnikov \cite{MR2487491}, are defined as follows:

A \emph{building set} on a finite set $E$ is a collection $\mathcal{B}$ of nonempty subsets of $E$ such that
\begin{itemize}
\item If $I,J \in \mathcal{B} $ and $ I \cap J \neq \emptyset$, then $I \cup J \in \mathcal{B}$,
\item $\{i\} \in \mathcal{B}$ for all $ i\in E$.
\end{itemize}
The \emph{nestohedron} $P_{\mathcal{B}} $ is defined as the Minkowski sum of simplices
\begin{align*}
P_{\mathcal{B}} =\sum_{I \in \mathcal{B}} \Delta_I \subset \R^E,
\end{align*}
where 
$\Delta_I \subset \R^E$ is the convex hull of $\{e_i\}_{i \in I}$ for the standard basis $\{e_i\}_{i \in E}$ of $\R^E$. 
By \cite[Proposition 3.12]{MR2191630}, \cite[Sections 6, 7]{MR2487491}, the nestohedron $P_{\mathcal{B}} $ 
coincides with the generalized permutohedron $P_f$ defined by the submodular function
\begin{align*}
f : 2^{E} \to \R, \   I \mapsto \# \calb -  \# \calb|_{I^o},
\end{align*}
where $\mathcal{B}|_{I^o} \coloneqq \{ J \in \mathcal{B} \mid   J \subset I^o \}$ is the \emph{restriction} of $\calb$ to $I^o\coloneqq E \setminus I$,
and $\# S$ denotes the cardinality of a set $S$.
Furthermore, the toric variety $X_{P_{\calb}}$ is smooth \cite[Proposition 7.10]{MR2487491}.

In \cite{MR4565222}, S.~Choi and T.~Hwang provided an explicit formula for the Gromov widths of symplectic toric manifolds defined by graph associahedra,
which form a subclass of nestohedra,
and raised a question regarding the case of general nestohedra \cite[Remark 4.2]{MR4565222}.

In this paper, we give an explicit formula for $\ep(X_P,L_P;1_P) $ and $w_G(X_P,\omega_P)\coloneqq w_G(X_P,\omega_{L_P})$
for a generalized permutohedron $P$, not only for nestohedra,
where $1_P \in T_M \subset X_P$ is the identity element of the maximal torus $T_M \simeq (\C^{\times})^n$.

To state the formula, we recall an invariant of convex bodies.
For a convex body $K \subset M_{\R}$, that is, a closed convex set with non-empty interior,
the \emph{lattice width} $\wid(K) $  is defined as
\[
\wid(K) =  \min_{v \in N \setminus \{0\} }\max_{u,u' \in K} \abs{ v(u) -v(u') } ,
\]
where we regard $v \in N \subset N_{\R} $ as a linear map $v : M_{\R} \to \R$. 
Equivalently, $\wid(K)$ is the minimum of the length $\abs{v(K)}$ of the segment  $v(K) \subset \R$ for $v \in N \setminus \{0\}$.
From the definitions of Seshadri constants and lattice widths,
it is straightforward to see that $\ep(X_P,L_P;1_P) \leq \wid(P)$ holds (see \autoref{rem_Ito14} (2)).
When $P$ is Delzant,
that is, when $X_P$ is smooth, the corresponding inequality
$w_G(X_P,\omega_P) \leq \wid(P)$ has recently been proved by \cite[Corollary 3.8]{MR4717262}.
Unlike $\ep(X_P,L_P;1_P)$ and $w_G(X_P,\omega_P)$, $\wid(P)$ is easy to compute
(e.g., by using \texttt{LATTICE\_WIDTH} in \textsf{polymake} \cite{MR1785292}).

The following is the main result of this paper:

\begin{thm}\label{thm_gen_permutohedra}
Let $P=P_f \subset \R^E$ be the generalized permutohedron defined by a submodular function $f : 2^E \to \R$ with $f(\emptyset)=0$.
Assume that $\dim P >0$.
Then it holds that
\begin{align}\label{eq_thm_gen_permutohedra}
\begin{split}
\ep(X_P,L_P;1_P) &=\wid(P) \\
&= \min \{ f(I) + f(I^o) -f(E)  \mid   I  \in 2^E,  f(I) + f(I^o) -f(E)  >0\}.
\end{split}
\end{align}

In particular, if $P$ is Delzant, 
it holds that
\[
w_G(X_P,\omega_P) = \min \{ f(I) + f(I^o) -f(E)  \mid   I  \in 2^E,  f(I) + f(I^o) -f(E)  >0\}.
\]
\end{thm}

We note that $\dim P \leq \# E-1$ holds for a generalized permutohedron  $P \subset \R^E$ and hence $P$ is not full-dimensional in $ \R^E$.
When we consider $X_P, L_P, \wid(P)$ in \autoref{thm_gen_permutohedra},
we regard $ P$ as a full-dimensional quasi-rational polytope in the affine space $\mathrm{Aff}(P) \subset \R^E$ spanned by $P$,
equipped with the affine lattice induced by $\Z^E$.
This makes sense since the isomorphism class of $(X_P,L_P)$ and $\wid(P)$ does not change by parallel translations.

As a corollary,
we give an answer to the question 
in \cite{MR4565222}:

\begin{cor}\label{cor_nestohedra}
Let $P=P_{\calb} \subset \R^E$ be the nestohedron defined by a building set $\calb$ on a finite set $ E$.
If $\dim P >0$, that is, if $\calb \neq \big\{ \{i\} \mid  i \in E \big\}$, 
then 
\begin{align*}
\ep(X_P,L_P ;1_P) =w_G(X_P,\omega_P) =\wid(P)= \min \{  n_I \mid   I \in \calb, n_I >0\},
\end{align*}
where $n_I\coloneqq  \# \{ J \in \calb \mid   I \cap J \neq \emptyset, I^o \cap J \neq \emptyset\} $ for $I \in 2^E$.
\end{cor}

As shown in \cite[Example 5.6]{MR4251298}, for a symplectic toric $4$-fold $(X,\omega)$ with moment polygon $P \subset \R^2$,
the strict inequality $w_G(X,\omega) < \wid(P)$ can occur.\footnote{On the other hand, we always have $\frac{3}{4} \wid(P) < w_G(X,\omega)$ for such $(X,\omega)$, and the constant $\frac34$ is sharp; see \cite[Corollary 1.2]{Ito:2024aa}.}
However, if $X$ is additionally Fano, the equality $w_G(X,\omega) = \wid(P)$ holds, as stated in the paragraph following the example.\footnote{In fact, they state that $w_G(X,\omega)$ coincides with the upper bound $\Upsilon (P)$ given by \cite[Theorem 1.2]{MR2210713} for such $(X,\omega)$.
By \cite[Proposition 3.9]{MR4533481}, the bound $\Upsilon (P)$ coincides with the lattice width $\wid (P)$ for a Delzant polytope $P$.}
In \cite[Question 5.9]{MR4251298},
the authors ask whether the equality $w_G(X,\omega) = \wid(P)$ still holds for toric Fano manifolds  in higher dimensions,
particularly in the case where $\omega$ represents the anti-canonical divisor $-K_X$.
The following corollary gives an affirmative answer to the question in real dimension $6$.

\begin{cor}\label{cor_fano_3fold}
Let $P \subset \R^3$ be a quasi-rational polytope such that $ X_P$ is a smooth toric Fano variety of dimension $3$.
Then $\ep(X_P,L_P;1_P) =w_G(X_P,\omega_P) =\wid (P)$.
\end{cor}

\vspace{2mm}
This paper is organized as follows. 
In \autoref{sec_pre},
we recall some bounds for Seshadri constants of toric varieties in \cite{MR3276156}. 
In \autoref{sec_gen_perm}, we prove \autoref{thm_gen_permutohedra}, \autoref{cor_nestohedra}.
In \autoref{sec_toric_fano_3fold}, we prove \autoref{cor_fano_3fold}.
In this paper, varieties are defined over the field of complex numbers $\C$.
The notation $A \subset B$ means that  $A$ is a subset of $B$.
We use the notation $A \subsetneq B$ if $A$ is a proper subset of $B$.

\subsection*{Acknowledgments}
The author would like to thank Professors Suyoung Choi and Taekgyu Hwang for useful comments.
The author was supported by JSPS KAKENHI Grant Numbers 21K03201, 26K06719.

\section{Preliminaries}\label{sec_pre}

First, we recall some notation from toric geometry. We refer the reader to \cite{MR2810322, MR1234037} for details.

Let $M \simeq \Z^n$ be a free abelian group and  $P \subset M_{\R}$ be a lattice polytope of dimension $n$.
Let $X_P$ be the toric variety defined by the normal fan $\Sigma_P$ of $P$.
Let $v_1,\dots, v_r$ be all the primitive vectors in $N$ which span rays in $\Sigma_P$.
Then we can write
\[
P=\{ u \in M_{\R} \mid   \langle v_i, u \rangle \geq -a_i \text{ for } 1 \leq i \leq r\}
\]
for some $a_i \in \Z$ and we define a divisor $L_P$ on $X_P$ by
\[
L_P= \sum_{i=1}^r a_i D_i,
\]
where $D_i \subset X_P$ is the toric prime divisor corresponding to the ray spanned by $v_i$.
It is known that $L_P$ is an ample Cartier divisor on $X_P$.

This construction can be generalized to the case where $P $ is a quasi-rational polytope.
In fact, the normal fan $\Sigma_P$ is rational as well, and hence we can define $X_P$ and $L_P$ as above for some $a_i \in \R$. 
In \cite[Section 4]{MR3466351}, it is shown that $L_P$ is an ample $\R$-Cartier $\R$-divisor.
Hence we can define
the Seshadri constant $\ep(X_P,L_P;1_P)$ by
\[
 \ep(X_P,L_P;1_P)=\inf_C \left\{\dfrac{(C \cdot L_P)}{\mult_{1_P}(C)}\right \}
\]
as in the case of lattice polytopes.
For simplicity, the Seshadri constant 
$\ep(X_P,L_P;1_P)$ is denoted by $\ep(P)$ in the rest of this section.

A key tool in this paper is the following result from \cite{MR3276156}.

\begin{thm}[{\cite[Theorem 3.6]{MR3276156}}]\label{thm_combinatorial_bound}
Let $M, M'$ be free abelian groups of rank $n, r$, respectively, and $ \pi: M_{\R} \to M'_{\R}$ be the linear map induced from a surjective group homomorphism $\pi_{\Z} : M \to M'$.
Let $P \subset M_{\R} $ be an $n$-dimensional 
quasi-rational polytope and take $u' \in \pi(P) $ such that the dimension of $ P(u')\coloneqq \pi^{-1}(u') \cap P$ is $n-r$.
Then we have
\begin{align}\label{eq_thm_Ito14}
\min\{ \ep(\pi(P)), \ep(P(u'))\} \leq \ep(P) \leq \ep(\pi(P)),
\end{align}
where we regard $ P(u') \subset \pi^{-1}(u')  $ as a quasi-rational polytope in $\ker \pi =(\ker \pi_{\Z} )_{\R} \simeq \R^{n-r}$ by a suitable parallel translation $\pi^{-1}(u')   \simeq \ker \pi $.
\end{thm}

\begin{rem}\label{rem_Ito14}
\begin{enumerate}
    \item The original statement of \cite[Theorem 3.6]{MR3276156} assumes that $P$ is a rational polytope and $u' \in \pi(P) \cap M'_{\Q}$. 
    The quasi-rational case follows from the rational case by approximating $P$ with rational polytopes and $u'$ with points in $\pi(P) \cap M'_{\Q}$.
    \item For a segment $Q \subset \R$, we have $(X_Q,L_Q)=(\P^1, \calo_{\P^1}(\abs{Q}))$, and hence
    \begin{align}\label{eq_1-dim}
    \ep(Q) = \ep(\P^1, \calo_{\P^1}(\abs{Q});1_Q) =\deg \calo_{\P^1}(\abs{Q}) =\abs{Q}.
    \end{align}
    Thus, if $r=\rank M' =1$, \ref{eq_thm_Ito14} is equivalent to
    \begin{align}\label{eq_r=1}
    \min\{ \abs{\pi(P)}, \ep(P(u'))\} \leq \ep(P) \leq \abs{\pi(P)}.
    \end{align}
    In particular, $\ep(P) \leq \wid(P)$ holds.
\end{enumerate}
\end{rem}

As a consequence of \autoref{thm_combinatorial_bound}, we obtain the following corollary.

\begin{cor}\label{cor_lower_bound}
Let $M \simeq \Z^n$ be a free abelian group, $P \subset M_{\R}$ be a quasi-rational polytope of dimension $n$,
and $(X_P,L_P)$ be the corresponding $\R$-polarized toric variety.

Assume that there exist
\begin{itemize}
\setlength{\itemsep}{0mm}
\item  $p_i,q_i \in P$ for $1 \leq i \leq n$,
\item $u_1, \dots, u_n \in M$, which form a basis of $M_{\R} \simeq \R^n$,
\item a real number $\rho > 0$
\end{itemize}
such that 
\begin{align}\label{eq_q_i-p_i}
q_i-p_i \in  \rho u_i +  \R u_{i+1} \oplus \cdots \oplus \R u_{n}
\end{align}
for each $1 \leq i \leq n$.\footnote{By convention,
we set $\R u_{i+1} \oplus \cdots \oplus \R u_{n} =\{0\}$ for $i=n$.
In particular, $ q_n-p_n=\rho u_n$ holds.}
Then $\ep(P) \geq  \rho$ holds.

In particular,
if $P$ is Delzant, 
then $w_G(X_P, \omega_{P}) \geq  \rho$  also holds, where $ \omega_P$ is a K\"ahler form on $X_P$ representing the class $c_1(L_P)$.
\end{cor}

\begin{proof}
For $p,q \in M_{\R}$, $[p,q] \subset M_{\R}$ denotes the line segment with endpoints $p,q$.

We show this corollary by induction on $n$.
If $n=1$, we have 
\[
\ep(P) =\abs{P} \geq  \abs{[p_1,q_1] } \geq \rho
\]
by \ref{eq_1-dim} since $q_1-p_1=\rho u_1$.

Assume $n \geq 2$ and let $\pi_{\Z} : M \to M'\coloneqq M/ (\R u_n \cap M)$ be the canonical quotient map,
which induces a linear map $\pi : M_{\R} \to M'_{\R}$.
Since a fiber of $\pi$ contains $[p_n,q_n] =[p_n,p_n+ \rho u_n ]$,
we have
\[
\ep(P) \geq \min\{  \ep(\pi(P)) , \rho\}
\]
by \ref{eq_r=1}.
Since
\begin{itemize}
\setlength{\itemsep}{0mm}
\item $\pi(p_i),\pi(q_i) \in \pi(P)$ for $1 \leq i \leq n-1$, 
\item $\pi(u_1), \dots, \pi(u_{n-1}) \in M'$, which form a basis of $M'_{\R} \simeq \R^{n-1}$
\end{itemize}
satisfy
\begin{align*}
\pi(q_i)-\pi(p_i) \in  \rho \pi(u_i) +  \R \pi(u_{i+1}) \oplus \cdots \oplus \R \pi(u_{n-1})
\end{align*}
for $1 \leq i \leq n-1$,
we have $\ep(\pi(P)) \geq \rho$ by the induction hypothesis.
Hence $\ep(P)  \geq \rho$ holds.

The last statement follows from the inequality $\ep(P) \leq w_G(X_P,\omega_P)$ in \cite[Corollary 2.1.D]{MR1262938}.
\end{proof}

\begin{rem}
In \cite{MR4565222},
the authors use lower bounds for Gromov widths obtained in \cite[Section 4.2]{MR3190299} and \cite[Proposition 5]{MR3763945};
see also \cite[Proposition 3.3]{MR4251298} for a detailed proof.
These results establish $ w_G(X_P,\omega_P) \geq \rho$ in the setting of \autoref{cor_lower_bound} under the following additional conditions:
\begin{enumerate}[label=(\roman*)]
\setlength{\itemsep}{0mm}
\item $q_i-p_i =\rho u_i$ instead of \ref{eq_q_i-p_i}, 
\item $u_1,\dots,u_n$ form a basis of $M$,
\item the line segments $[p_i,q_i]$ share a common point.
\end{enumerate}
The proof in \cite{MR4251298} explicitly constructs symplectic embeddings of open balls into $X_P$,
where condition (iii) is essential for identifying the image of the center of the ball.

Although Seshadri constants can be interpreted as the maximal size of  a K\"ahler embedding of open balls (see, for instance, \cite{MR3656421, MR3761104}),
we do not use this interpretation in the proof of \autoref{thm_combinatorial_bound} in \cite{MR3276156}.
Instead, we directly estimate the intersection numbers appearing in \ref{eq_def_SC}.
As a consequence, \autoref{cor_lower_bound} does not require condition (iii),
which increases the flexibility in choosing $p_i$ and $q_i$.
For an example illustrating this advantage, see \autoref{ex_HLS}.
This flexibility is also essential in the proof of \autoref{thm_gen_permutohedra} in \autoref{sec_gen_perm}.

We also note that \autoref{cor_lower_bound} generalizes the case $j=n$ of \cite[Lemma 4.6]{MR4065714},
which establishes $ \ep(P) \geq \rho$ under the additional condition (i).
\end{rem}

\begin{ex}\label{ex_quadrilateral}
For $ t \in \Q$, let $P_t =\Conv( (1,t), (0,1), (-1,-t), (0,-1) ) \subset \R^2$ as in \autoref{fig:polygon}. 
By taking 
\[
p_1=(-1,-t), \  q_1=(1,t), \  p_2 =(0,-1), \  q_2=(0,1), \  u_1=(1,0), \  u_2=(0,1), \  \rho=2 ,
\]
we have $\ep(P_t) \geq 2$.
Since $ \wid(P_t) \leq \abs{p_1 (P_t)} =2$ for the projection $p_1 : \R^2 \to \R$ to the first factor,
we have $ \ep(P_t) =\wid(P_t)=2$.
Of course, we can directly apply  \autoref{thm_combinatorial_bound} to the projection $p_1 : \R^2 \to \R$.
\end{ex}

\begin{figure}[htbp]
  \centering
 \includegraphics[scale=0.7]{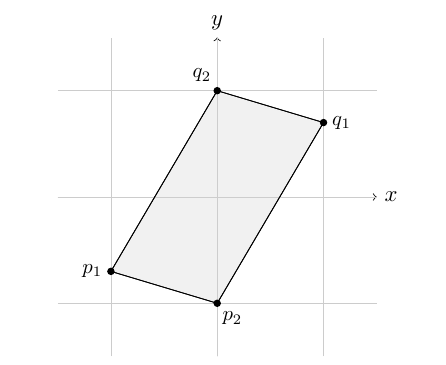}
  \caption{The parallelogram $P_t$ in \autoref{ex_quadrilateral}} 
  \label{fig:polygon}
\end{figure}

\begin{ex}\label{ex_HLS}
For $t > 0$, let $P_t \subset \R^3$ be the convex hull of
\begin{align*}
[-1, -1+t] \times [0,1] \times \{1\} \ \cup \ [-1, t] \times [-1,1] \times \{0\} \ \cup \ [-1, 1+t] \times [-1,0] \times \{-1\}
\end{align*}
as in \autoref{fig:HLS}.
We note that $P_t$ coincides with the moment polytope of the symplectic toric manifold
$(X, \omega_t)$ in \cite[Example~5.7]{MR4251298} up to $\GL(\Z^3)$.
In \cite[Example 5.7, Question 5.8]{MR4251298}, the authors show $\frac{5+t}{3} \leq w_G(X_{P_t},\omega_{P_t}) \leq 2$ for $0 < t \leq 1$
and ask what the exact value of $w_G(X_{P_t},\omega_{P_t})$ is.
We can show that $\ep(P_t)= w_G(X_{P_t},\omega_{P_t}) = 2$ for any $t >0$ as follows:

By $\ep(P_t) \leq w_G(X_{P_t},\omega_{P_t}) \leq \wid(P_t) =2$, it suffices to show $\ep(P_t) \geq 2 $.
This follows from \autoref{cor_lower_bound} since 
\begin{align*}
[p_1,q_1] &=[-1,1] \times \{0\} \times \{-1\},  \\
[p_2,q_2] &=  \{-1\} \times [-1,1 ] \times \{0\}, \\
[p_3,q_3]& =  \{-1\} \times \{0\} \times [-1,1]
\end{align*}
and 
\[
u_1=(1,0,0), \quad u_2=(0,1,0), \quad u_3=(0,0,1), \quad \rho=2
\]
satisfy the assumption in \autoref{cor_lower_bound}.
\end{ex}

\begin{figure}[htbp]
  \centering
 \includegraphics[scale=0.88]{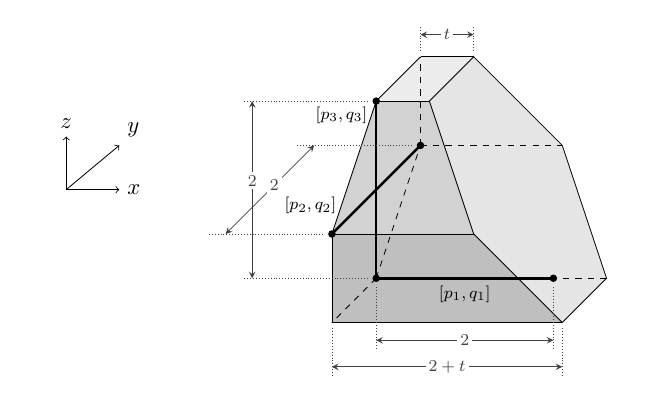}
  \caption{The polytope $P_t$ in \autoref{ex_HLS}} 
  \label{fig:HLS}
\end{figure}

\begin{ex}\label{ex_toric_fano3fold}
The polytope $P_1$ in \autoref{ex_HLS} corresponds to the anti-canonical divisor $-K_X$ of a smooth toric Fano $3$-fold $X$.
It is known that there are $18$ smooth toric Fano $3$-folds, which are  classified by \cite{MR631434, MR670903}.
We can compute $\ep(X,-K_X;1)$ and $w_G(X,\omega_{-K_X})$ with moment polytope $P$ for such Fano $3$-folds as follows:  
\begin{enumerate}
\setlength{\itemsep}{0mm}
\item If $X=\P^3$, then $\ep(X,-K_X;1)=w_G(X,\omega_{-K_X})=\wid(P) =4$.
\item If $X=\P_{\P^1}(\calo \oplus\calo \oplus \calo(1))$,
then $\ep(X,-K_{X}; 1) = w_G(X, \omega_{-K_X}) =\wid (P)= 3$.
\item For the remaining $16$ cases, $\ep(X,-K_{X}; 1) = w_G(X, \omega_{-K_X}) =\wid (P)= 2$.
\end{enumerate}
For example, the moment polytope $P$ of $-K_X$ in (2) is given by
\begin{align*}
P=\Conv (\{-1\} \times \{-1\} \times [-1,3], \{2\} \times \{-1\} \times [-1,0], \{-1\} \times \{2\} \times [-1,0] ) 
\end{align*}
(see \autoref{fig:Fano3fold}).
Hence, the inequalities $3 \leq \ep(X,-K_{X}; 1) \leq w_G(X, \omega_{-K_X}) \leq \wid(P)$ hold by taking
\begin{align*}
[p_1,q_1]&=[-1,2] \times \{-1\} \times \{-1\} , \\
[p_2,q_2] &= \{-1\} \times [-1,2] \times \{-1\}, \\
[p_3,q_3] &= \{-1\} \times \{-1\} \times [-1,2].
\end{align*}
On the other hand, $ \wid (P) \leq 3$ holds  since $p_1(P) =[-1,2]$ for the projection $p_1 : \R^3 \to \R$ to the first factor.
We leave (1), (3) to the reader.

In particular,  $\ep(X,-K_{X}; 1) = w_G(X, \omega_{-K_X}) =\wid (P)$ holds for any smooth toric Fano $3$-fold $X$.
We will generalize this statement for any ample $\R$-line bundle $L$ on such $3$-fold $X$ in \autoref{cor_fano_3fold}.
\end{ex}

\begin{figure}[htbp]
  \centering
 \includegraphics[scale=0.88]{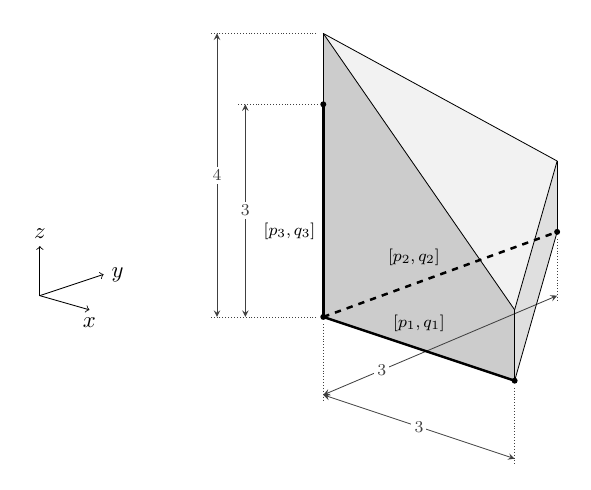}
  \caption{The moment polytope of $-K_X$ for $X=\P_{\P^1}(\calo \oplus\calo \oplus \calo(1))$}
  \label{fig:Fano3fold}
\end{figure}

\section{Proof of \autoref{thm_gen_permutohedra}}\label{sec_gen_perm}

In this section, we prove \autoref{thm_gen_permutohedra} and \autoref{cor_nestohedra}.
We say that a submodular function $f : 2^E \to \R$ with $f(\emptyset) =0$ is \emph{connected}
if $f(I) + f(I^o) > f(E)$ for any $ \emptyset \neq I \subsetneq E$.
To show \autoref{thm_gen_permutohedra}, we recall some results on submodular functions.

\begin{lem}[{\cite[Section 3.3]{MR2171629}}]\label{lem_GP}
Let $f : 2^E \to \R$ be a submodular function with $f(\emptyset) =0$
and  $P_f \subset \R^E$ be the generalized permutohedron defined by \ref{eq_gen_permutohedra}.
Then there exists a unique partition\footnote{A \emph{partition} means that 
$E_1,\dots, E_c$ are nonempty and pairwise disjoint, and satisfy $E=E_1 \cup E_2 \cup \cdots \cup E_c$.}
$E=E_1 \sqcup E_2 \sqcup \cdots \sqcup E_c$ 
 such that 
\begin{enumerate}
\setlength{\parskip}{1mm} 
\item for each $1 \leq k \leq c$,
the restriction $f|_{2^{E_k}} : 2^{E_k} \to \R$, which is also submodular, is connected.
\item $f(I) = f(I \cap E_1) + \cdots + f(I \cap E_c)$ for any $I \in 2^E$.
\item $P_f = P_{f_1} \times \cdots \times P_{f_c} \subset \R^{E_1} \times \cdots \times \R^{E_c}=\R^E$.
\item $\dim P = \# E -c$.
\end{enumerate}
\end{lem}

\begin{proof}
We note that $P_f$ is denoted by $B(f)$ in  \cite{MR2171629}.
The existence and uniqueness of  the partition $E=E_1 \sqcup E_2 \sqcup \cdots \sqcup E_c$ which satisfies (1), (2)  are proved in \cite[Theorem 3.38]{MR2171629}.
Since $f(E)=f(E_1) + \cdots +f(E_c)$ by (2),
the conditions 
\begin{align*}
\sum_{i \in E} x_i = f(E) , \quad \sum_{i \in E_k} x_i \leq  f(E_k) \quad \text{for } 1 \leq k \leq c
\end{align*} 
are equivalent to
\begin{align*}
\sum_{i \in E_k} x_i =  f(E_k) \quad \text{for } 1 \leq k \leq c.
\end{align*} 
Furthermore,
the condition $\sum_{i \in I} x_i \leq f(I)$ for $I \in 2^E$ follows from
 $\sum_{i \in I \cap E_k} x_i \leq f(I \cap E_k) $ for $1 \leq k \leq c$ by (2).
 Hence $P_f$ coincides with the polytope  defined by
 \begin{align*}
\sum_{i \in E_k} x_i =  f(E_k) , \quad \sum_{i \in I} x_i \leq f(I) \quad \text{for } 1 \leq k \leq c \text{ and } I \in 2^{E_k},
\end{align*}
which is nothing but $P_{f_1} \times \cdots \times P_{f_c}$.
 (4) is proved in  \cite[Corollary 3.40]{MR2171629}.
\end{proof}

In order to apply \autoref{cor_lower_bound} to $P_f$,
we use the following description of the Minkowski sum  $P_f -P_f$.

\begin{lem}\label{lem_P-P_GP}
Let $f : 2^E \to \R$ be a submodular function with $f(\emptyset) =0$. 
Then the Minkowski sum $P_f-P_f \coloneqq P_f+(-P_f)= \{ x -y \mid x, y \in P_f\}$ coincides with
\begin{align}\label{eq_P_f-P_f}
\set*{ x \in \R^E \given  \sum_{i \in E} x_i =0, \  \sum_{i \in I } x_i    \leq  f(I) +f(I^o) -f(E)  \text{ for any } I \in 2^{E} }.
\end{align}
\end{lem}

\begin{proof}
We see that $-P_f =\{ -x \mid x \in P_f \}$ is a generalized permutohedron defined by
\[
g :  2^E \to \R, \  I \mapsto f(I^o) -f(E) 
\]
as follows.
It is straightforward to check that $g$ is submodular with $g(\emptyset) = 0$.
On the other hand,
\begin{align*}
-P_f=\set*{ x \in \R^E \given -\sum_{i \in E} x_i=f(E),\  -\sum_{i \in I } x_i  \leq f(I) \text{ for any } I \in 2^{E} }
\end{align*}
by \ref{eq_gen_permutohedra}.
Under the condition $ -\sum_{i \in E} x_i=f(E)$, it holds that
\begin{align*}
-\sum_{i \in I } x_i  \leq f(I) \ \lra \ \sum_{i \in I^o } x_i  +f(E) \leq f(I) \ \lra \ \sum_{i \in I^o } x_i  \leq f(I) -f(E) .
\end{align*}
Hence 
\begin{align*}
-P_f&=\set*{ x \in \R^E \given \sum_{i \in E} x_i=- f(E), \ \sum_{i \in I^o } x_i  \leq f(I) -f(E)\text{ for any }  I \in 2^{E}  }\\
&=\set*{ x \in \R^E \given \sum_{i \in E} x_i=- f(E), \ \sum_{i \in I } x_i  \leq f(I^o) -f(E)\text{ for any }  I \in 2^{E}  } =P_g.
\end{align*}

Then, by \cite[Section 4.2]{MR2171629}, 
$P_f - P_f =P_f +P_g$ coincides with $P_{f+g} $, which is precisely \ref{eq_P_f-P_f}.
\end{proof}

Now we can prove  \autoref{thm_gen_permutohedra}.

\begin{proof}[Proof of \autoref{thm_gen_permutohedra}]
The last statement concerning $w_G(X_P,\omega_P)$ follows from $ \ep(X_P,L_P;1_P) \leq w_G(X_P,\omega_P) \leq \wid(P)$.
Hence it suffices to show that both $\ep(X_P,L_P;1_P) $ and $\wid(P)$ coincide with
\begin{align*}
l \coloneqq  \min \{ f(I) + f(I^o) -f(E)  \mid   I  \in 2^E,  f(I) + f(I^o) -f(E)  >0\}.
\end{align*}

\vspace{2mm}
First,
we assume that the function $ f$ is connected,
that is, $f(I) +f(I^o) - f(E) >0$ for any $I \in 2^E \setminus \{ \emptyset, E\}$.
In this case, $\dim P =\# E -1$ by \autoref{lem_GP} (4).
For $I \in 2^E \setminus \{ \emptyset, E\}$,
let $\pi_I : \R^E \to \R$  be the linear map defined by $\pi_I(x) =\sum_{i \in I} x_i$.
Since $I \neq \emptyset, E$ and $\dim P=\# E-1$,  we have $\dim \pi_I(P) = 1$.
Since $f(E) -f(I^o) \leq \sum_{i \in I} x_i \leq f(I) $ for any $x \in P$,
we have 
\begin{align*}
 \wid(P) \leq \abs{\pi_I(P) } \leq f(I) +f(I^o) - f(E).
\end{align*}
Hence we have $\ep(X_P,L_P; 1_P) \leq \wid(P) \leq l$.

Let $\{e_i\}_{i \in E}$ be the standard basis of $\R^E$
and fix $i_0 \in E $.
By \autoref{lem_P-P_GP} and the connectedness of $f$,
the element $l(e_i-e_{i_0}) $ is contained in $ P -P$ for each $i \in E \setminus\{i_0\}$ since 
$\pi_I(l(e_i-e_{i_0})) \in\{-l,0,l\}$
for any $I \in 2^E$.
Hence we can take $p_i, q_i \in P$ such that
$q_i -p_i =l(e_i-e_{i_0})$.
Since $\{e_i-e_{i_0}\}_{i \in  E\setminus \{i_0\}}$ is a basis of $ \{x \in \R^E \mid   \sum_{i \in E} x_i =0 \} \simeq \R^{\# E-1}$,
we can apply \autoref{cor_lower_bound} and obtain $\ep(X_P,L_P; 1_P) \geq l $.
This proves the theorem in the connected case.

\vspace{2mm}
In the general case, we take the partition $E=E_1 \sqcup \cdots \sqcup E_c$ in \autoref{lem_GP}.
For quasi-rational polytopes $Q_1 \subset \R^{n_1}, Q_2 \subset \R^{n_2} $,
we have $(X_{Q_1 \times Q_2}, L_{Q_1 \times Q_2}) =(X_{Q_1} \times X_{Q_2}, \pr_1^* L_{Q_1} \otimes \pr_2^* L_{Q_2})$ and hence
\[
\ep(X_{Q_1 \times Q_2}, L_{Q_1 \times Q_2};1_{Q_1 \times Q_2}) = \min\{ \ep(X_{Q_1 }, L_{Q_1};1_{Q_1}), \ep(X_{Q_2}, L_{ Q_2};1_{Q_2})\}
\]
by \cite[Proposition 3.4 (e)]{MR3338009}.
Furthermore, we can  easily check the equality $\wid(Q_1 \times Q_2 ) =\min \{\wid(Q_1), \wid(Q_2) \}$.
Thus both $\ep(X_P,L_P;1_P) $ and $ \wid(P)$ coincide with
\begin{align*}
 l' \coloneqq \min_{1 \leq k \leq c} \min\{ f(J) +f(E_k \setminus J) -f(E_k ) \mid    J \in 2^{E_k} \setminus \{\emptyset, E_k\} \}
\end{align*}
by \autoref{lem_GP} (3) and the connected case of this theorem.
Hence it suffices to show that $l=l'$.

For $I \in 2^E$, set $I_k\coloneqq I \cap E_k$.
Then it holds that
\begin{align*}
f(I) &= f(I_1) + \cdots +f(I_c),\\
f(I^o) &= f(I^o \cap E_1) + \cdots +f(I^o \cap E_c) =f(E_1 \setminus I_1) + \cdots + f(E_c \setminus I_c),\\
f(E) &=f(E_1) + \cdots +f(E_c)
\end{align*}
by \autoref{lem_GP} (2).
Consequently, we have
\begin{align}\label{eq_not_connected1}
f(I) +f(I^o) -f(E) = \sum_{k=1}^c  \left( f(I_k) +f(E_k \setminus I_k) -f(E_k ) \right).
\end{align}
For each $1\leq k \leq c$, we have $f(I_k) +f(E_k \setminus I_k) -f(E_k ) \geq 0$,
and the strict inequality holds if and only if $I_k \neq \emptyset, E_k$ by the connectedness of $f|_{2^{E_k}}$.
Hence $l \geq l'$ holds.

On the other hand,  if $J \in 2^{E_k} \setminus \{\emptyset, E_k\} $ for some $k$, then 
\begin{align*}
f(J) +f(J^o) -f(E) =  f(J) +f(E_k \setminus J) -f(E_k )  >0
\end{align*}
by \ref{eq_not_connected1}.
Thus $l \leq l'$, and hence $l = l'$, completing the proof.
\end{proof}

\begin{ex}
For real numbers $a_1 \geq  \cdots \geq a_n $,
the \emph{permutohedron} $P_n=P_n(a_1, \dots, a_n)  \subset \R^{n}$ is defined to be the convex hull of 
\[
\{(a_{\sigma(1)}, \dots, a_{\sigma(n)}) \in \R^{n} \mid  \text{$\sigma$ is a permutation on $[n] \coloneqq \{1,2, \dots, n\}$} \} .
\]
By \cite{MR45168},
 $P_n(a_1, \dots, a_n) $ is the generalized permutohedron defined by
\begin{align*}
f : 2^{[n]} \to \R  \ : \ I \mapsto  a_1 + \cdots +a_{\# I}.
\end{align*}
Since $P_n$ is a singleton when $a_1=a_n$,
we assume $a_1  >a_n$.     
Then we have
\begin{align*}
f(I) + f(I^o) -f([n]) &=(a_1 + \cdots + a_{\# I }) +( a_1 + \cdots + a_{\# I^o } )-(a_1+ \cdots + a_n) \\
&=(a_1 + \cdots + a_{\# I }) -  (  a_{\# I^o + 1} + \cdots + a_n)\\
&=a_1-a_n + \sum_{i=2}^{\# I} (a_i - a_{\# I^o +i-1})\\
&\geq a_1-a_n
= f(\{1\}) + f(\{1\}^o) - f([n])
\end{align*}
for $I \in 2^{[n]} \setminus \{\emptyset, [n]\}$.
Thus it holds that
\begin{align*}
\ep(X_{P_n}, L_{P_n};1_{P_n})=\wid(P_n)= a_1-a_n.
\end{align*}
If $X_{P_n}$ is smooth, $w_G(X_{P_n}, \omega_{P_n} ) = a_1-a_n$ also holds,
as shown in \cite[Remark 4.1]{MR4565222} in the case $a_1 >  \cdots > a_n $.
\end{ex}

\begin{ex}\label{ex_Minkowski_sum}
Let $P$ be the Minkowski sum
\[
P=\sum_{I \in 2^E} y_I \Delta_I, 
\]
where $y_I\ge 0$ and $\Delta_I\subset \R^E$ denotes the convex hull of $\{e_i\}_{i\in I}$.
Assume that $\dim P>0$.

By \cite[Proposition 3.12]{MR2191630}, \cite[Section 6]{MR2487491},
$P$ is the generalized permutohedron $P_f$ associated with
\begin{align*}
f : 2^E \to \R, \  I \mapsto \sum_{J \in 2^E, J \cap I \neq \emptyset} y_J .
\end{align*}
Since 
\begin{align*}
f(I) +f(I^o) -f(E) &=  \sum_{J \in 2^E, J \cap I \neq \emptyset} y_J +  \sum_{J \in 2^E, J \cap I^o \neq \emptyset} y_J - 
 \sum_{J \in 2^E \setminus \{\emptyset\}} y_J\\
 &=\sum_{J \in 2^E , J \cap I \neq \emptyset, J \cap I^o \neq \emptyset}  y_J   \eqqcolon w_I,
\end{align*}
we have 
\[
\ep(X_P,L_P;1_P) =\wid (P)  = \min\{ w_I \mid I \in 2^E, w_I >0\}.
\]
\end{ex}

Now we can prove the following corollary, which specializes to \autoref{cor_nestohedra} when $y_I=1$ for all $I \in \calb$.

\begin{cor}\label{cor_Minkowski_sum}
Let $\calb$ be a building set on a finite set $E$ and 
$P=\sum_{I \in \calb} y_I \Delta_{I}$ for $y_I \in \R_{> 0}$.
If $\dim P >0$, equivalently if $\calb \neq \{ \{i\} \mid  i \in E\}$, 
then $X_P$ is smooth and
\[
\ep(X_P,L_P;1_P)= w_G(X_P,\omega_P) =\wid (P) = \min  \left\{w_I \mid  I \in \calb, w_I >0 \right\}
\]
for $w_I =\sum\limits_{J \in \calb , J \cap I \neq \emptyset, J \cap I^o \neq \emptyset} y_J $.
\end{cor}

\begin{proof}
Since  $y_I >0$ for any $I \in \calb$,
the normal fan $\Sigma_P$ of $P$ coincides with that of the nestohedron $P_\calb$.
Hence  $X_P = X_{P_{\calb}}$ is smooth.
By setting $y_I \coloneqq 0$ for $I \in 2^E \setminus \calb$,
we can apply  \autoref{ex_Minkowski_sum}.
Hence it suffices to show that the minimum of
\begin{align*}
\set*{w_I = \sum\limits_{J \in \calb , J \cap I \neq \emptyset, J \cap I^o \neq \emptyset} y_J  \given I \in 2^E, w_I >0}
\end{align*}
is attained by
some $I \in \calb \subset 2^E$.

Let $E_1, \dots, E_c $ be all the maximal elements in $\calb$.
By the definition of building sets,
$E_1, \dots, E_c $ are nonempty, pairwise disjoint, and $E=E_1 \cup \cdots \cup E_c$,
that is, 
$E=E_1 \sqcup \cdots \sqcup E_c$ is a partition.\footnote{We can check that this partition is nothing but that in \autoref{lem_GP},
though we do not use this fact.}

Fix $I \in 2^E$ which satisfies $w_I >0$, and
set $I_k \coloneqq I \cap E_k$ for $1 \leq k \leq c$.
What we need to show is that there exists $S \in \calb$ such that $ w_I \geq w_S >0$.

Since each $J \in \calb$ is contained in just one of $E_1, \dots, E_c $,
it holds that
\begin{align*}
\{ J \in \calb \mid   J \cap I \neq \emptyset, J  \cap  I^o  \neq \emptyset\} 
=\bigsqcup_{k=1}^c \{ J \in \calb \mid  J \cap I_k  \neq \emptyset, J \cap  I_k^{o}  \neq \emptyset\}
\end{align*}
and hence $w_I = \sum_{k=1}^c w_{I_k}$.
Since we assume $w_I >0$, there exists $1 \leq  k \leq c$ such that $w_I \geq w_{I_k} >0$.
By renumbering the indices, we may assume that $w_I \geq w_{I_1} >0$.

Let $ S_1, \dots, S_r $ be all the maximal elements in $\calb|_{I_1}=\{ J \in \calb \mid    J \subset I_1\}$.
Since $w_{I_1} >0$, $I_1$ is nonempty.
Thus $\calb|_{I_1}$ is nonempty since it contains $\{i\}$ for $i  \in I_1$,
and hence $r \geq 1$.

\begin{claim}\label{claim_nesto}
$w_{I_1} \geq w_{S_1} >0$ holds.
\end{claim}

\begin{proof}[Proof of \autoref{claim_nesto}]
If $J \in \calb$ satisfies $J \cap E_1 \neq \emptyset$, then $J \cup E_1 \in \calb$  by the definition of building sets and hence $J \subset E_1$ by the maximality of $E_1$ in $\calb$.
Thus we have $w_{E_1} =0$.
Hence $I_1 \subsetneq E_1$ by $w_{I_1} >0$.
Then $E_1 \in \calb$ satisfies $E_1 \cap S_1  = S_1 \neq \emptyset$ and $E_1 \cap S_1^o  \supset E_1 \cap I_1^o \neq \emptyset$.
Hence we have $w_{S_1} \geq y_{E_1} >0$.

To see that $w_{I_1} \geq w_{S_1} $, it suffices to show 
\begin{align*}
\{ J \in \calb \mid   J \cap I_1  \neq \emptyset, J \cap I_1^o  \neq \emptyset\} \supset 
\{ J \in \calb \mid   J \cap S_1  \neq \emptyset, J \cap S_1^o  \neq \emptyset\} .
\end{align*}
Assume $ J \in \calb$ satisfies $J \cap S_1  \neq \emptyset, J \cap S_1^o  \neq \emptyset$.
Then $J \cap I_1  \supset  J \cap S_1 \neq \emptyset$.
If $J \cap I_1^o  = \emptyset $, then $J \subset I_1$.
Since $J , S_1 \in \calb|_{I_1}$ satisfy $J \cap S_1  \neq \emptyset$,
we have $ J \subset S_1$ by the maximality of $S_1$ in the building set $\calb|_{I_1}$,
which contradicts $J \cap S_1^o  \neq \emptyset$.
Thus  $J \cap I_1^o  \neq \emptyset $ and the above inclusion holds.
\end{proof}

Hence we have $ w_I \geq w_{I_1} \geq w_{S_1} >0$ for $S_1 \in \calb|_{I_1} \subset \calb$,
and this corollary is proved.
\end{proof}

\begin{rem}
For a Delzant polytope $P \subset M_{\R}$,
it is known that $\wid(P)$ coincides with the \emph{facet width},
that is,
the minimum of $ \{ \abs{\pi_F(P)}  \mid  \text{$F$ is a facet of $P$}\}$,
where $\pi_F : M_{\R} \to \R$ is the linear map defined by the primitive inner normal vector to the facet $F$ 
\cite[Lemma 3.11]{MR4533481}.
Hence if a generalized permutohedron $P_f$ is Delzant,
then we may take the minimum in \ref{eq_thm_gen_permutohedra} only for $I \in 2^E$ which defines a facet of $ P_f$.

In particular, $\wid (P) = \min \, \{ w_I \mid  I \in \calb, w_I >0 \} $ in \autoref{cor_Minkowski_sum} also follows from \autoref{ex_Minkowski_sum} and the description of faces of nestohedra \cite[Corollary 3.13]{MR2191630},  \cite[Theorem 7.4]{MR2487491}.
\end{rem}

\begin{ex}\label{ex_CH}
Let $\calb=\big\{\{1\}, \{2\}, \{3\}, \{4\}, \{1,2\}, \{3,4\}, \{1,2,3,4\}\big\}$ be the building set on $E=\{1,2,3,4\}$ in \cite[Remark 4.2]{MR4565222}.
Let $P=\sum_{I \in \calb} y_I \Delta_{I}$ with $y_I \in \R_{> 0}$.
For $ I \in \calb$, it holds that 
\begin{align*}
\{J \in \calb \mid  J \cap I \neq \emptyset, J \cap I^o \neq \emptyset\}  =
\begin{cases}
\big\{\{1,2\}, \{1,2,3,4\}\big\} &  \text{if $I=\{1\}$ or $\{2\}$,}  \\
\big\{\{3,4\}, \{1,2,3,4\}\big\} &  \text{if $I=\{3\}$ or $\{4\}$,}  \\
\big\{\{1,2,3,4\}\big\} &  \text{if $I=\{1,2\}$ or $\{3,4\}$,}  \\
\ \emptyset &  \text{if $I=\{ 1,2,3,4\}$} .
\end{cases}
\end{align*}
Hence we have
\[
\ep(X_P,L_P;1_P)= w_G(X_P,\omega_P) =\wid (P) =y_{\{1,2,3,4\}}.
\]
In particular,
$\ep(X_{P_{\calb}},L_{P_{\calb}};1_{P_{\calb}})= w_G(X_{P_{\calb}},\omega_{P_{\calb}}) =\wid ({P_{\calb}}) =1 $ holds.
\end{ex}

\begin{figure}[htbp]
  \centering
 \includegraphics[scale=0.69]{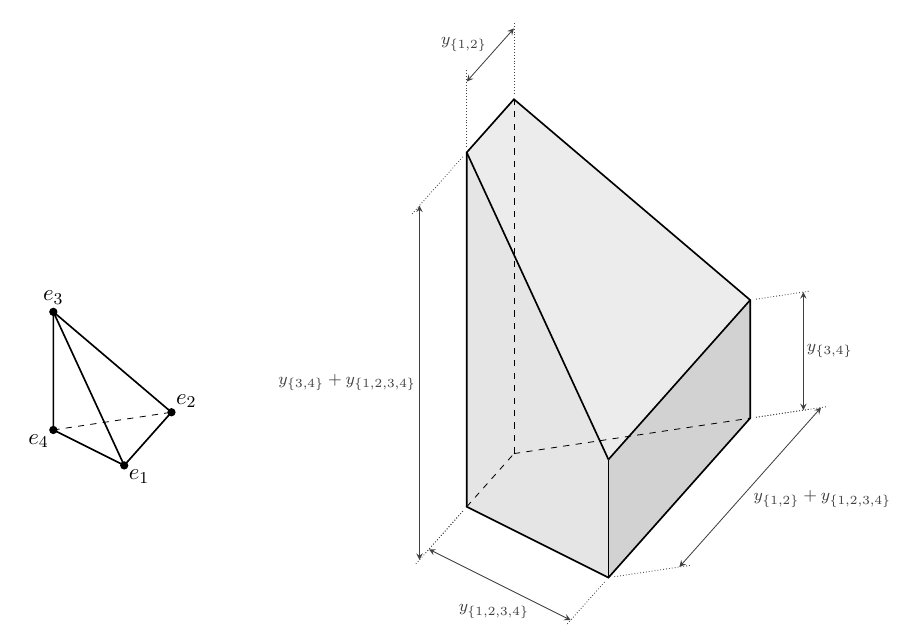}
\caption{The nestohedron $P=\sum_{I \in \calb} y_I \Delta_{I}$ in \autoref{ex_CH}}
  \label{fig:CH}
\end{figure}

\section{Smooth toric Fano $3$-folds}\label{sec_toric_fano_3fold}

\begin{proof}[Proof of \autoref{cor_fano_3fold}]
Among the $18$ smooth toric Fano $3$-folds classified by \cite{MR631434, MR670903}, 
there are two $3$-folds $X_1, X_2$ with Picard number $5$.
By \cite[Example 2.6]{MR4065180},
there exist building sets $\calb_1, \calb_2$ such that $X_{P_{\calb_i}} = X_i$ for each $i=1,2$.
It is known that a quasi-rational polytope $P \subset \R^E$ is a generalized permutohedron
if and only if every edge of $P$ is parallel to $e_i-e_j$ for some $i,j \in E$ (see \cite[Theorem 3.1.3]{MR4651496} for example).
If a toric variety $X$ is obtained from $X_i = X_{P_{\calb_i}}$ by successive blow-downs along toric strata,
then every edge of the moment polytope $Q$ of any ample $\R$-line bundle on $X$ is parallel to an edge of the nestohedron $P_{\calb_i}$.
Hence $Q$ is a generalized permutohedron.

Among the $18$ smooth toric Fano $3$-folds,
$16$ of them (including $X_1,X_2$) are obtained from $X_1$ or $X_2$ by successive blow-downs along toric strata (see \cite[\S 1]{MR670903}).
If $X_P$ is one of these 16 types,
it holds that $\ep(X_P,L_P;1_P) =w_G(X_P, \omega_P) =\wid(P)$ by \autoref{thm_gen_permutohedra}.

The remaining two $3$-folds are $ Y_1=\P_{\P^2}( \calo \oplus \calo(2))$ and $Y_2 $, which is obtained by blowing up a line $l \subset  \P_{\P^2}( \calo ) \subset  \P_{\P^2}( \calo \oplus \calo(2))$.
Assume $X_P=Y_1$.
An ample $\R$-line bundle on $ Y_1=\P_{\P^2}( \calo \oplus \calo(2))$ is of the form $\calo_{\varpi}(a) \otimes \varpi^* \calo_{\P^2}(b) $ for $a,b >0$,
where $ \varpi: \P_{\P^2}( \calo \oplus \calo(2)) \to \P^2$ is the canonical morphism and $\calo_\varpi(1)$ is the tautological quotient line bundle.
The moment polytope $P_{a,b} \subset \R^3$ of $\calo_{\varpi}(a) \otimes \varpi^* \calo_{\P^2}(b) $  is the convex hull of 
\begin{align*}
\{0\} \times b\Delta   \quad \cup \quad \{a\} \times (2a+b)\Delta ,
\end{align*}
where $\Delta =\Conv ((0,0),(1,0), (0,1)) \subset \R^2$. 
Applying \autoref{thm_combinatorial_bound} to the projection $p_1 : \R^3 \to \R$ to the first factor and $u'= a \in \R$,
we have $\ep( X_{P_{a,b}},L_{P_{a,b}};1_{P_{a,b}} ) =a$
since $p_1(P_{a,b}) =[0,a]$ and $ p_1^{-1} (a) \cap P_{a,b}= (2a+b)\Delta$.
Since
\[
 \wid (P_{a,b}) \leq \abs{ p_1(P_{a,b})} =a =\ep( X_{P_{a,b}},L_{P_{a,b}};1_{P_{a,b}} ) \leq w_G ( X_{P_{a,b}},\omega_{P_{a,b}})   \leq  \wid (P_{a,b}),
\]
we obtain $\ep( X_{P_{a,b}},L_{P_{a,b}};1_{P_{a,b}} ) = w_G ( X_{P_{a,b}},\omega_{P_{a,b}})  = \wid (P_{a,b}) =a$.

Assume $X_P=Y_2$.
An ample $\R$-line bundle on $ Y_2$ is of the form $ \mu^* L_{P_{a,b} }  - c E$ for $a,b,c >0$ with $\min\{a,b\} >c$,
where $\mu : Y_2 \to Y_1$ is the blow-up along the line $l$ corresponding to the edge $\Conv ((0,0,0),(0,0,b)) \subset P_{a,b}$ and $E \subset Y_2$ is the exceptional divisor.
The moment polytope $P_{a,b,c}$ of $ \mu^* L_{P_{a,b} }  - c E$ is the intersection 
\[
P_{a,b} \cap \{ (x,y,z) \in \R^3  \mid x+y \geq c\}.
\]
Since $p_1(P_{a,b,c}) =[0,a]$ and $ p_1^{-1} (a) \cap P_{a,b,c} = (2a+b)\Delta$,
we obtain $\ep( X_{P_{a,b,c}},L_{P_{a,b,c}};1_{P_{a,b,c}} ) = w_G ( X_{P_{a,b,c}},\omega_{P_{a,b,c}})  = \wid (P_{a,b,c}) =a$
by the same argument as for $P_{a,b}$.
\end{proof}

\bibliographystyle{amsalpha}
\bibliography{mainbibs}

\end{document}